\documentclass[12pt]{article}

\usepackage{amsmath}
\usepackage{mathtools}
\usepackage{amssymb}
\usepackage{fullpage}
\usepackage{mathrsfs}
\usepackage{enumerate}
\usepackage{multirow}
\usepackage{algpseudocode}
\usepackage{algorithm}
\usepackage{subfigure}
\usepackage{color}
\usepackage{cite}
\usepackage{dsfont}
\usepackage{slashbox}

\usepackage{float}

\usepackage{newfloat}
\DeclareFloatingEnvironment[
    fileext=lob,
    listname=List of Boxes,
    name=Box,
    placement=tbhp,
]{mybox}

\definecolor{lblue}{RGB}{0,110,152}

\definecolor{dred}{RGB}{171,67,53}

\everymath{\displaystyle}

\newtheorem{theorem}{Theorem}%[\arabic{section}]

\newtheorem{proposition}[theorem]{Proposition}

\newtheorem{define}[theorem]{Definition}

\newtheorem{remark}[theorem]{Remark}

\DeclareMathOperator*{\col}{col}

\DeclareMathOperator{\He}{Sym}
\DeclareMathOperator*{\diag}{diag}

\DeclareMathOperator{\eps}{\varepsilon}

\def\E{\mathbb{E}}
\def\P{\mathbb{P}}

\newenvironment{proof}{{\it Proof :~}}{\hfill$\diamondsuit$\\}

\begin{document}

\title{Stability analysis and state-feedback control of LPV systems with piecewise constant parameters subject to spontaneous Poissonian jumps}

\author{Corentin Briat\thanks{Corentin Briat is an independent researcher. email: corentin@briat.info; url: www.briat.info}}

\date{}

%\markboth{IEEE TRANSACTIONS ON AUTOMATIC CONTROL, VOL. XX, NO. X, XX XXXX}%
%{Stability analysis and state-feedback control of LPV systems with piecewise constant parameters subject to spontaneous jumps}

\maketitle

\begin{abstract}
LPV systems with piecewise constant parameters subject to spontaneous Poissonian jumps are a class of systems that does not seem to have been thoroughly considered in the literature. We partially fill this gap here by providing sufficient stability and performance analysis conditions stated in terms of infinite-dimensional LMI problems that can be solved using sum of squares programming. A particularity of the obtained conditions lies in  the presence of an integral term leading to some technical difficulties when attempting to obtain convex conditions for the design of a gain-scheduled state-feedback controller. This difficulty is circumvented by relying on a recent result for time-delay systems analysis and an equivalent integral-free LMI condition is obtained. The approach is illustrated through several examples.
\end{abstract}

%\begin{IEEEkeywords}
%LPV systems; stochastic hybrid systems; stochastic parameters; sum of squares programming
%\end{IEEEkeywords}
%\IEEEpeerreviewmaketitle

\section{Introduction}

Linear parameter-varying systems (LPV systems) \cite{Mohammadpour:12,Briat:book1} are a particular class of systems able to represent a wide variety of real-world processes \cite{Mohammadpour:12} ranging from automotive applications \cite{Poussot:10,Sename:13}, aperiodic sampled-data systems \cite{Robert:10} to aerospace systems \cite{Barker:00,Shin:00}, etc. The main advantage of LPV systems lies in the possibility of designing, in a very systematic way, gain-scheduled controllers \cite{Shamma:88phd,Shamma:92,Briat:14e}. Many results have been obtained in this respect; see e.g. \cite{Packard:94a,Apkarian:95,Apkarian:95a,Apkarian:98a,Wu:01,Scherer:01,Scherer:12,Briat:14e,Briat:15d,Briat:17ifacLPV,Briat:17AutomaticaLPV} and references therein. In \cite{Briat:15d}, LPV systems with piecewise constant parameters were considered and several stability analysis and stabilization conditions were proposed. The jumps in the parameter trajectories were assumed to be fully deterministic and to obey some minimum dwell-time condition \cite{Geromel:06b,Goebel:12}. In fact, LPV systems with piecewise constant parameters are a natural generalization of linear switched systems where the mode of the system now takes values within some interval instead of in some finite set. This was later generalized to parameter trajectories that are piecewise differentiable in \cite{Briat:17ifacLPV,Briat:17AutomaticaLPV}. In the vast majority of the existing works, however, the parameters are systematically assumed to evolve deterministically and, this is only very recently, that some results on stochastic LPV/uncertain systems have started to appear; see e.g. \cite{Nagira:15,Hosoe:18}.

The objective of the paper is to extend the framework of \cite{Briat:15d} to the case where the jumps in the parameter trajectories now occur spontaneously and in a Poissonian way (i.e. the time between two successive jumps is exponentially distributed). The resulting system can be written as a piecewise deterministic Markov process or a stochastic hybrid system \cite{Davis:93,Teel:14} where the flow part consists of the deterministic dynamics of the state and the parameters of the LPV system and the jump part consists of the stochastic update rule for the parameters (both the next jump time and the next value of the parameter). Interestingly, while the framework in \cite{Briat:15d} generalized the framework of switched systems to uncountable mode values, the current one generalizes the framework of Markov jump linear systems \cite{Costa:13} to the case where the mode also takes values in some uncountable bounded set.% In the case, where countability is preserved but the set of values for the mode is still infinite, then we would recover the framework developed in \cite{Todorov:08b}.

The contributions of the paper are as follows. Firstly, a sufficient condition for the mean-square (exponential) stability of LPV systems with piecewise constant parameters subject to spontaneous Poissonian jumps is provided. It is based on the use of a parameter-dependent quadratic Lyapunov function that is reminiscent of that used in the context of deterministic LPV systems. Interestingly, this condition combines the flow (deterministic) and the jump (stochastic) parts in a single one. This has to be contrasted with the case of deterministic hybrid systems where two separate conditions, one for the flow part and one for the jump part, are usually obtained \cite{Goebel:12,Briat:15d,Briat:17ifacLPV,Briat:17AutomaticaLPV}. The condition takes the form of a parameter-dependent LMI that has the particularity of involving an integral term, which may be a potential source of difficulty. Assuming that the data of the system and that the infinite-dimensional decision variables in the LMI conditions are polynomial, the integral can be explicitly evaluated and the resulting conditions can be solved using polynomial methods such as sum of squares programming; see e.g. \cite{Parrilo:00,Chesi:10b}. It is important to note here that restricting the data and the decision variables to be polynomials is quite reasonable since the parameters are assumed to take value within a compact set and that any continuous function on a compact set can be approximated by a polynomial with arbitrary precision. This approach, as opposed to gridding methods, is exact in the sense that all the possible parameter values are considered and, in spite of that, it may still offer a more tractable alternative in situations where, for instance, the semidefinite program obtained from gridding exceeds in size that of  the SOS one. In both cases, however, the computational complexity may grow very quickly with the size of the system, the number of parameters and the complexity of the geometry of the parameter set, and may rapidly result in intractable programs. This is an intrinsic limitation of all the approaches based on parameter-dependent LMIs, but there is, to date, no viable alternative.

 The stability condition is then extended to capture an $L_2$ performance criterion, thereby providing a Bounded Real Lemma for this class of systems for the first time. In order to be able to exploit the latter stability condition for control design in an efficient way (i.e. using nonconservative manipulations), it is necessary to get rid of the integral term from the LMI condition. This can be achieved by relying on a result from \cite{Peet:09} that stipulates that a certain slack variable can be added to the condition so that the integral term can be removed. This allows for the derivation of convex design conditions at the expense of an increase of the computational complexity of the problem.

\noindent\textbf{Outline.} The structure of the paper is as follows. Preliminary definitions and results are given in Section \ref{sec:prel}. Section \ref{sec:stab} is devoted to the stability and performance analysis of LPV systems with piecewise constant parameters subject to spontaneous Poissonian jumps. These results are then extended to state-feedback design in Section \ref{sec:stabz}. Computational discussions are provided in Section \ref{sec:comp} and examples in Section \ref{sec:ex}.

\noindent\textbf{Notations.} The cone of symmetric (positive definite) matrices of dimension $n$ is denoted by $\mathbb{S}^n$ ($\mathbb{S}^n_{\succ0}$). For $A,B\in\mathbb{S}^n$, the expression $A\prec(\preceq)B$ means that $A-B$ is negative (semi)definite. The operator $\partial_x$ means differentiation with respect to the variable $x$. For some square matrix $A$, we define $\He[A]=A+A^T$. The Lebesgue measure of a compact set $\mathcal P$ is denoted by $\mu(\mathcal{P})$.

%In what follows, $(\Omega,\cal F,\P)$ denotes a complete probability space where

\section{Preliminaries}\label{sec:prel}

We consider in this paper LPV systems with stochastic piecewise constant parameters subject to spontaneous Poissonian jumps described by the following stochastic hybrid system with flow
\begin{equation}\label{eq:mainsyst}
\begin{array}{rcl}
    \dot{x}(t)&=&A(\rho(t))x(t)+B(\rho(t))u(t)+E(\rho(t))w(t)\\
    z(t)&=&C(\rho(t))x(t)+D(\rho(t))u(t)+F(\rho(t))w(t)\\
    %y(t)&=&C_y(\rho(t))x(t)+F_y(\rho(t))w(t)\\
    x(0)&=&x_0
\end{array}
\end{equation}
where $x,x_0\in\mathbb{R}^n$, $\rho(t)\in\mathcal{P}$, $u\in\mathbb{R}^m$, $w\in\mathbb{R}^p$, $z\in\mathbb{R}^q$ and $y\in\mathbb{R}^r$  are the state of the LPV system, the initial condition, the parameter of the LPV system, the control input, the disturbance and the controlled output, respectively. We assume that the parameters are piecewise constant (i.e. $\dot\rho=0$ in between jumps) and randomly change their values according to a jump process with finite jump intensity. That is,  we have that
\begin{equation}\label{eq:mainsystR}
  \P[\rho(t+h)\in B|\rho(t)=\rho]=\kappa(\rho,B)h+o(h)
\end{equation}
where $\rho\in\mathcal{P}$, $B\subseteq\mathcal{P}-\{\rho\}$ is measurable and where $\kappa:\mathcal{P}\times\mathcal{P}\to\mathbb{R}_{\ge0}$ is the \emph{instantaneous jump rate} such that $\rho\mapsto\kappa(\rho,A)$ is measurable and $A\mapsto\kappa(\rho,A)$ is a positive measure. More specifically, $\kappa(\rho,d\theta)$ are the transition rates whereas $\bar\lambda(\rho)=\textstyle\int_\mathcal{P}\kappa(\rho,d\theta)$ are the intensities. It is clear from its definition that the process $(x(t),\rho(t))_{t\ge0}$ is a Markov process. For simplicity, we assume from now on that $\kappa(\rho,d\theta)=\lambda(\rho,\theta)d\theta$ where $\lambda$ is a polynomial function. %We will further assume that $\lambda$ is a polynomial function. Note that this is not restrictive in the present case as polynomials can approximate a continuous function on a compact set with arbitrary accuracy.
Similarly, we assume that the matrices in \eqref{eq:mainsyst} are polynomial functions of $\rho$, making the overall system a polynomial stochastic hybrid system; see e.g. \cite{Hespanha:06c}.

The following proposition states a fundamental result describing the evolution of the process  \eqref{eq:mainsyst}-\eqref{eq:mainsystR}:
\begin{proposition}
  The infinitesimal generator $\mathbb{A}$ of the Markov process $(x(t),\rho(t))_{t\ge0}$ defined by \eqref{eq:mainsyst}-\eqref{eq:mainsystR} with $u,w\equiv0$ is given by
    \begin{equation}\label{eq:generator}
    \mathbb{A}f(\rho,x):=\partial_x  f(x,\rho)A(\rho)x+\int_{\mathcal{P}}\lambda(\rho,\theta)[f(x,\theta)-f(x,\rho)]d\theta
  \end{equation}
  where $f$ is any bounded function.
\end{proposition}
The Dynkin's formula is given in this case by
\begin{equation}\label{eq:dynkin}
  \E[f(x(\tau),\rho(\tau))]=f(x(0),\rho(0))+\int_0^\tau\E[(\mathbb{A}f)(x(s),\rho(s))]ds
\end{equation}
for any stopping time $\tau$.
\begin{define}
  The system \eqref{eq:mainsyst}-\eqref{eq:mainsystR} is said to be mean-square stable if for any initial condition $(x_0,\rho_0)$, we have that $\E[||x(t)||_2^2]\to0$ as $t\to\infty$. It is said to be mean-square exponentially stable with rate $\alpha>0$ if there exists a $\beta\ge1$ such that $\E[||x(t)||_2^2]\le\beta e^{-2\alpha t}||x(0)||_2^2$.
\end{define}
\begin{define}
The $L_2$-norm of a signal $w:[0,\infty)\mapsto\mathbb{R}^n$ is defined as
\begin{equation}
  ||w||_{L_2}:=\left(\int_0^\infty \E||w(s)||_2^2ds\right)^{1/2}.
\end{equation}
When $||w||_{L_2}<\infty$, we say that $w\in L_2$.
\end{define}
\begin{define}
  The (stochastic) $L_2$-gain of the map $L_2\ni w\mapsto z\in L_2$ as induced by the system \eqref{eq:mainsyst}-\eqref{eq:mainsystR} with $u\equiv0,x_0=0$ is defined as
  \begin{equation}
    ||w\mapsto z||_{L_2-L_2}:=\sup_{||w||_{L_2}=1}||z||_{L_2}.
  \end{equation}
\end{define}

%\subsection{Statistical properties of the parameters}

%
%\begin{theorem}
%  The probability density function $p_t(\rho)$ of the parameters obeys the differential equation
%  \begin{equation}
%   \begin{array}{rcl}
%         \dfrac{\partial p_{\rho_0}(\rho,t)}{dt}&=&-\bar{\lambda}(\rho)p_{\rho_0}(\rho,t)+\int_{\mathcal{P}}\lambda(\rho,\theta)p_{\rho_0}(\theta,t)d\theta\\
%         p_{\rho_0}(\rho,0)&=&\delta_{\rho_0}(\rho)
%   \end{array}
%  \end{equation}
%  where $\bar{\lambda}(\rho):=\int_{\mathcal{P}}\lambda(s,\rho)ds$. We, moreover, have that
%  \begin{equation}
%    \P[\rho(t)\in I|\rho(0)=\rho_0]=\int_I p_{\rho_0}(s,t)ds,\ I\subset\mathcal{P}.
%  \end{equation}
%\end{theorem}

%
%\begin{theorem}
%  Let $f:\mathcal{P}\to\mathbb{R}$ be a continuous function. Then, we have that
%  \begin{equation}
%    \dfrac{d\E[f(\rho(t))]}{dt}=-\E[\bar{\lambda}(\rho(t))f(\rho(t))]+\int_{\mathcal{P}}\E[\lambda(\rho(t),\theta)]f(\theta)d\theta.
%  \end{equation}
%\end{theorem}
%\begin{proof}
%  Since the state-space of the Markov process describing the evolution of the parameters is compact, then the set of continuous functions depending only on the parameters belongs to the domain of the generator. The result then follows from the definition of the generator to this special case.
%\end{proof}

%The above equation is difficult to solve in the general case. This  due to the so-called moment closure problem

\section{Stability analysis and $L_2$-performance}\label{sec:stab}

In this section, we address first the problem of establishing the (exponential) mean-square stability of the process \eqref{eq:mainsyst}-\eqref{eq:mainsystR}

\subsection{Asymptotic and exponential stability}

We first address the problem of establishing the asymptotic/exponential mean-square stability of the system  \eqref{eq:mainsyst}-\eqref{eq:mainsystR}. This is formulated in the following result:
\begin{theorem}\label{th:stab}
  Assume that there exists a matrix-valued function $P:\mathcal{P}\mapsto\mathbb{S}_{\succ0}^n$ and a scalar $\alpha>0$ such that the LMI
  \begin{equation}\label{eq:stab}
    \He[P(\rho)A(\rho)]+\int_{\mathcal{P}}\lambda(\rho,\theta)[P(\theta)-P(\rho)]d\theta+2\alpha P(\rho)\preceq0
  \end{equation}
  holds for all $\rho\in\mathcal{P}$. Then, the system \eqref{eq:mainsyst}-\eqref{eq:mainsystR} with $u\equiv0$ and $w\equiv0$ is mean-square exponentially stable with rate $\alpha$.
\end{theorem}
\begin{proof}
  To prove this result, let us consider the function $V(x,\rho)=x^TP(\rho)x$ where $P(\cdot)$ is as in the result. The feasibility LMI condition implies that $\mathbb{A}V(x,\rho)\le-2\alpha V(x,\rho)$.
  %\begin{equation}
%      \mathbb{A}V(x,\rho)\le-2\alpha V(x,\rho)
%  \end{equation}
  Hence, we have that
    \begin{equation}
  \dfrac{d}{dt}\E[V(x(t),\rho(t))]\le-2\alpha\E[V(x(t),\rho(t))]
  \end{equation}
  and, therefore, $\E[V(x(t),\rho(t))]\le V(x(0),\rho(0))e^{-2\alpha t}$.
  %\begin{equation}
%  \E[V(x(t),\rho(t))]\le V(x(0),\rho(0))e^{-2\alpha t}.
%  \end{equation}
  Since $\bar\lambda_m||x||_2^2\le V(x,\rho)\le \bar\lambda_M||x||_2^2$ for all $x\in\mathbb{R}^n$ where $\bar\lambda_m:=\min_{\rho\in\mathcal P}\lambda_{min}(P(\rho))$ and $\bar\lambda_M:=\max_{\rho\in\mathcal P}\lambda_{max}(P(\rho))$, then $\bar\lambda_m\E[||x||_2^2]\le \E[V(x,\rho)]\le \bar\lambda_M\E[||x||_2^2]$ and we get that
  \begin{equation}
  \E[||x(t)||_2^2]\le \beta e^{-2\alpha t}||x_0||_2^2,\ \beta=\bar\lambda_M/\bar\lambda_m
  \end{equation}
  which proves the result.
\end{proof}
%
%As is often the case for linear systems, exponential stability and asymptotic stability are equivalent. Indeed, assuming that the condition \eqref{eq:stab} holds with $\alpha=0$ and with an open inequality sign, then there exist a sufficient small $\eps>0$ such that \eqref{eq:stab} holds with $\alpha=\eps$. The above result is interesting in the sense that it shows that unlike in the deterministic case, the stability of the LPV system does not imply that the matrix $A(\rho)$ be Hurwitz stable; see e.g. \cite{Briat:book1,Briat:15d}. Indeed, it is only necessary that the matrix $A(\rho)-\bar{\lambda}(\rho)I$ be Hurwitz stable where $\textstyle\bar{\lambda}(\rho)=\int_\mathcal{P}\lambda(\rho,\theta)d\theta$. Another difference with the deterministic conditions for hybrid systems is that the flow (deterministic) and jump (stochastic) parts are combined in the same LMI condition whereas they are in two separate ones in the deterministic case. Finally, the integral term is of particular nature and will need to be carefully taken care of. Since it is assumed that both $P(\cdot)$ and $\lambda(\cdot,\cdot)$ are polynomials, then integration can be performed very easily as polynomials are stable by integration.

%As is often the case for linear systems, exponential stability and asymptotic stability are equivalent. Indeed, assuming that the condition \eqref{eq:stab} holds with $\alpha=0$ and with an open inequality sign, then there exist a sufficient small $\eps>0$ such that \eqref{eq:stab} holds with $\alpha=\eps$.
The above result is interesting in the sense that it shows that unlike in the deterministic case, the stability of the LPV system does not imply that the matrix $A(\rho)$ be Hurwitz stable; see e.g. \cite{Briat:book1,Briat:15d}. Indeed, it is only necessary that the matrix $A(\rho)-\bar{\lambda}(\rho)I$ be Hurwitz stable where $\textstyle\bar{\lambda}(\rho)=\int_\mathcal{P}\lambda(\rho,\theta)d\theta$. %Another difference with the deterministic conditions for hybrid systems is that the flow (deterministic) and jump (stochastic) parts are combined in the same LMI condition whereas they are in two separate ones in the deterministic case.
The integral term is of particular nature and will need to be carefully taken care of. Since it is assumed that both $P(\cdot)$ and $\lambda(\cdot,\cdot)$ are polynomials, then integration can be performed very easily as polynomials are stable by integration.

\subsection{$L_2$-performance - Bounded Real Lemma}

We now extend the asymptotic/exponential stability result to account for an $L_2$-performance characterization:
\begin{theorem}\label{th:stabL2}
Assume that there exist a matrix-valued function $P:\mathcal{P}\mapsto\mathbb{S}_{\succ0}^n$ and a scalar $\gamma>0$ such that the LMI \eqref{eq:stabL2}
\begin{figure*}
\vspace{3mm}
  \begin{equation}\label{eq:stabL2}
    \begin{bmatrix}
      \He[P(\rho)A(\rho)]+\int_{\mathcal{P}}\lambda(\rho,\theta)[P(\theta)-P(\rho)]d\theta+C(\rho)^TC(\rho) & P(\rho)E(\rho)+C(\rho)^TF(\rho)\\
      \star & -\gamma^2I_p+F(\rho)^TF(\rho)
    \end{bmatrix}\prec0
  \end{equation}
  \end{figure*}
  holds for all $\rho\in\mathcal{P}$. Then, the system \eqref{eq:mainsyst}-\eqref{eq:mainsystR} with $u,w\equiv 0$ is mean-square stable and the $L_2$-gain of the operator $w\mapsto z$ is smaller than $\gamma$.
\end{theorem}
\begin{proof}
  From Theorem \ref{th:stab}, the condition of Theorem \ref{th:stabL2} implies the mean-square stability of the system  \eqref{eq:mainsyst}-\eqref{eq:mainsystR} whenever $u$ and $w$ are identically zero. Pre- and post-multiplying the LMI \eqref{eq:stabL2} by $\col(x(t),w(t))^T$ and $\col(x(t),w(t))$,  and invoking Dynkin's formula yields
  \begin{equation}
    %\E[V(x(t),\rho(t))]-V(x(0),\rho(0))+\E\left[\int_0^t\left(||z(s)||_2^2-\gamma^2||w(s)||_2^2\right)ds\right]\le0.
    \E[V(x(t),\rho(t))]+\E\left[\int_0^t\left(||z(s)||_2^2-\gamma^2||w(s)||_2^2\right)ds\right]\le0
  \end{equation}
  where we have assumed that $x_0=0$. Since the system is mean-square stable and $w\in L_2$, then we have that $\E[V(x(t),\rho(t))]\to0$ as $t\to\infty$ and we obtain
  \begin{equation}
    ||z||_{L_2}^2<\gamma^2||w||_{L_2}^2%+V(x(0),\rho(0))
  \end{equation}
  which implies that the $L_2$-gain of the map $w\mapsto z$ is less than $\gamma$. The proof is completed.
\end{proof}

\section{Stabilization with $L_2$-performance by state-feedback}\label{sec:stabz}

In this section, we aim at obtaining convex conditions for the design of a gain-scheduled state-feedback controller of the form
\begin{equation}\label{eq:sf}
  u=K(\rho)x
\end{equation}
where $K:\mathcal{P}\mapsto\mathbb{R}^{m\times n}$ is the parameter-dependent gain. The goal is to find constructive sufficient stabilization conditions for the gain $K(\cdot)$ such that the closed-loop system \eqref{eq:mainsyst}-\eqref{eq:mainsystR}-\eqref{eq:sf} is mean-square stable in the absence of disturbance $w$ and that the map $w\mapsto z$ has a guaranteed $L_2$-gain of at most $\gamma$. This is stated in the following result:
\begin{theorem}\label{eq:stabzL2}
  Assume that there exist matrix-valued functions $Q:\mathcal{P}\mapsto\mathbb{S}_{\succ0}^n$, $U:\mathcal{P}\mapsto\mathbb{R}^{m\times n}$ and $Z:\mathcal{P}\times\mathcal{P}\mapsto\mathbb{S}^n$ such that
  \begin{equation}
    \int_{\mathcal{P}}Z(\rho,\theta)d\theta=0
  \end{equation}
  holds for all $\rho\in\mathcal{P}$ and such that the LMI \eqref{eq:LMIstabzL2}
  \begin{figure*}[!t]
\normalsize
 \begin{equation}\label{eq:LMIstabzL2}
  \begin{bmatrix}
    \He[A(\rho)Q(\rho)+B(\rho)U(\rho)]-\bar{\lambda}(\rho)Q(\rho)+Z(\rho,\theta)& E(\rho) & (C(\rho)Q(\rho)+D(\rho)U(\rho))^T & \mu(\mathcal{P})\lambda(\rho,\theta)^{1/2}Q(\rho)\\
    \star &  -\gamma^2I_p & F(\rho)^T & 0\\
    \star & \star & -I_q & 0\\
     \star & \star & \star & -\mu(\mathcal{P})P(\theta)
  \end{bmatrix}\prec0
  \end{equation}
%% Restore the current equation number.
%%\setcounter{equation}{\value{mytempeqncnt}}
%% IEEE uses as a separator
%\hrulefill
%% The spacer can be tweaked to stop underfull vboxes.
%\vspace*{4pt}
\end{figure*}
%
%
%
%
%
%  \begin{figure*}[H]
%    \begin{equation}\label{eq:LMIstabzL2}
%  \begin{bmatrix}
%    \He[A(\rho)Q(\rho)+B(\rho)U(\rho)]-\bar{\lambda}(\rho)Q(\rho)+Z(\rho,\theta)& E(\rho) & (C(\rho)Q(\rho)+D(\rho)U(\rho))^T & \mu(\mathcal{P})\lambda(\rho,\theta)^{1/2}Q(\rho)\\
%    %
%    \star &  -\gamma^2I_p & F(\rho)^T & 0\\
%    \star & \star & -I_q & 0\\
%     \star & \star & \star & -\mu(\mathcal{P})P(\theta)
%  \end{bmatrix}\prec0
%  \end{equation}
%  \end{figure*}
  holds for all $\rho,\theta\in\mathcal{P}$ where $\mu(\mathcal{P})$ is the Lebesgue measure of the set $\cal P$. Then, the closed-loop system \eqref{eq:mainsyst}-\eqref{eq:mainsystR}-\eqref{eq:sf} is mean-square stable  in the absence of disturbance $w$ and the $L_2$-gain of the map $w\mapsto z$ is at most $\gamma$.
\end{theorem}
\begin{proof}
  First make the substitutions $A\leftarrow A_{cl}:=A+BK$ and $C\leftarrow C_{cl}:=C+DK$ in the LMI condition \eqref{eq:stabL2}. In an attempt to obtain a convex condition in the controller gain, we perform the standard congruence transformation with respect to the matrix $\diag(Q(\rho),I_p)$ where $Q(\rho)=P(\rho)^{-1}$ and, after a Schur complement,  we get that
    \begin{equation}
    \begin{bmatrix}\label{eq:stabL2b}
      \He[A_{cl}(\rho)Q(\rho)]+\mathcal{I} & E(\rho) & Q(\rho)C_{cl}(\rho)^T\\
      \star & -\gamma^2I_p & F(\rho)^T\\
      \star & \star & -I_q
    \end{bmatrix}\prec0
  \end{equation}
  where
  $$\mathcal{I}=-\bar\lambda(\rho)Q(\rho)+\int_{\mathcal{P}}\lambda(\rho,\theta)Q(\rho)P(\theta)Q(\rho)d\theta.$$ The difficulty here lies at the level of the integral term. If this integral was a sum as in standard Markov jump linear systems, then simple successive Schur complements will make the condition an LMI. In order to solve this problem, we use a result from \cite{Peet:09} that stipulates that the condition \eqref{eq:stabL2b} holds for all $\rho\in\mathcal{P}$ if and only if  there exists a matrix-valued function $Z:\mathcal{P}\times\mathcal{P}\mapsto\mathbb{S}^n$ verifying $\textstyle\int_{\mathcal{P}}Z(\rho,\theta)d\theta=0$ for all $\rho\in\mathcal{P}$ such that
      \begin{equation}
    \begin{bmatrix}\label{eq:stabL2b}
      \He[A_{cl}(\rho)Q(\rho)]+\mathcal{I}' & E(\rho) & Q(\rho)C_{cl}(\rho)^T\\
      \star & -\gamma^2I_p & F(\rho)^T\\
      \star & \star & -I_q
    \end{bmatrix}\prec0
  \end{equation}
  holds for $\rho,\theta\in\mathcal{P}$ where $$\mathcal{I}'=-\bar\lambda(\rho)Q(\rho)+\mu(\mathcal{P})\lambda(\rho,\theta)Q(\rho)P(\theta)Q(\rho)+Z(\rho,\theta).$$ A Schur complement together with the change of variables $U(\rho)=K(\rho)Q(\rho)^{-1}$ yield the result.
  \end{proof}

\section{Computational considerations}\label{sec:comp}

We discuss here some computational aspects related to solving the LMI conditions and to the simulation of the system.

\subsection{Solving for the LMI conditions}

The conditions formulated in the results in the previous sections are infinite-dimensional semidefinite programs and can not be solved directly. To make them tractable, we propose to consider an approach based on sum of squares programming \cite{Parrilo:00} that will result in an approximate finite-dimensional semidefinite program which can then be solved using standard solvers such as SeDuMi \cite{Sturm:01a}. The conversion to a semidefinite program can be performed using the package SOSTOOLS \cite{sostools3} to which we input the SOS program corresponding to the considered conditions. We say that a parameter-dependent symmetric matrix $M(\cdot)$ is a sum of squares matrix, if there exists a matrix $L(\cdot)$ such that $M(\cdot)=L(\cdot)^TL(\cdot)$. Obviously, a sum of squares matrix is positive semidefinite.

We illustrate below how an SOS program associated with some given conditions can be obtained. We assume here that the compact set $\mathcal{P}$ is semialgebraic, that is, it can be described as
\begin{equation}
  \mathcal{P}=:\left\{\theta\in\mathbb{R}^N: g_{i}(\theta)\ge0, i=1,\ldots,M\right\}
\end{equation}
where $g_i:\mathbb{R}^N\mapsto\mathbb{R}$, $i=1,\ldots,M$, are some known polynomials. This class notably encompass boxes, ellipsoids, etc. We have the following result:
\begin{proposition}\label{prog:periodic}
  Let $\eps,\alpha>0$ be given and  assume that the sum of squares program in Box \ref{box} is feasible. Then, the conditions of Theorem \ref{th:stab} hold with the computed polynomial matrix $P(\rho)$. As a result, the system \eqref{eq:mainsyst}-\eqref{eq:mainsystR} with $u,w\equiv0$ is exponentially mean-square stable with rate $\alpha$.
\end{proposition}
\begin{proof}
  The condition (C1) implies that the matrices $\Gamma_1^i(\rho),\Gamma_2^i(\rho)$ are positive semidefinite for all $\rho\in\mathbb{R}^N$ and all $i=1,\ldots,M$. The condition (C2) implies that $\textstyle P(\rho)-\sum_{i=1}^M\Gamma_1^i(\rho)g_i(\rho)-\eps I_n\succeq0$ for all $\rho\in\mathbb{R}^N$ and, hence, that $\textstyle P(\rho)\succeq \sum_{i=1}^M\Gamma_1^i(\rho)g_i(\rho) +\eps I_n$  for all $\rho\in\mathbb{R}^N$ and that $P(\rho)\succeq\eps I_n\succ0$ for all $\rho\in\mathcal{P}$. This proves the positive definiteness condition for the matrix $P(\rho)$. Similarly, the condition (C3) implies that the LMI \eqref{eq:stab} holds for all $\rho\in\mathcal{P}$. The proof is completed.
\end{proof}
\begin{remark}
  The integral equality constraint in Theorem \eqref{eq:stabzL2} on $Z$ can be easily implemented in the SOS program as it is simply equality constraints on the coefficients on the matrix polynomial $Z$.
\end{remark}

It may be interesting to ask whether the SOS conditions are also necessary in the sense that if the conditions stated in Theorem \ref{th:stab} are feasible, then a solution to the SOS program exists. Interestingly, this is true under some mild conditions whenever $N=1$ and $M=1$ (i.e. single parameter case with single constraint). This is a consequence of the facts that, in the one-parameter/one-constraint case, we have the equivalence between a symmetric matrix being SOS and its positive definiteness, and that we can approximate any continuous function by a polynomial with arbitrary precision. See \cite{Briat:14f} and the references therein for a more detailed explanation.

\begin{mybox*}
\vspace{2mm}
\centering
\caption{SOS program associated with Theorem \ref{th:stab}}\label{box}
{\vspace{1mm}}
\noindent\fbox{
\parbox{0.85\textwidth}{
    Find polynomial matrices $P,\Gamma_1^i,\Gamma_2^i:\mathcal{P}\mapsto\mathbb{S}^n$, $i=1,\ldots,M$, such that
      \begin{enumerate}[(C1)]
        \item $\Gamma_1^i(\rho),\Gamma_2^i(\rho)$ are SOS matrices for all $i=1,\ldots,M$;
        \item $P(\rho)-\sum_{i=1}^M\Gamma_1^i(\rho)g_i(\rho)-\eps I_n$ is an SOS matrix;
        \item $-\He[P(\rho)A(\rho)]-\int_{\mathcal{P}}\lambda(\rho,\theta)[P(\theta)-P(\rho)]d\theta-\sum_{i=1}^M\Gamma_2^i(\rho)g_i(\rho)-2\alpha P(\rho)$ is an SOS matrix.
    \end{enumerate}}}
\end{mybox*}

\subsection{Simulating the system}

The difficulty here lies in the fact that the parameters evolve in a stochastic manner. It seems important to briefly describe here how parameter trajectories can be obtained for simulation purposes. First of all, it is interesting to note that jumps are Poissonian, that is, the time between two consecutive jumps is exponentially distributed with rate $\bar\lambda(\rho)$ where $\rho$ is the current value of the parameter. By virtue of the memoryless property of the exponential distribution, it can be easily decided when the next jump will occur. Deciding what will be the next parameter value is slightly more involved, yet possible. Assume that $N=1$ and that $\mathcal{P}=[0,\bar{\rho}]$. Define further the following cumulative distribution function
\begin{equation}
  F_\rho(\theta):=\dfrac{1}{\bar\lambda(\rho)}\int_0^\theta\lambda(\rho,s)ds,\ \theta\in[0,\bar\rho].
\end{equation}
To obtain the next value for the parameter assuming that the current value for the parameter is $\rho$, we simply need to pick a value $\nu$ that is uniformly distributed in the interval $[0,1]$, and choose the next value for the parameter, $\rho^+$, that solves the expression $F_\rho(\rho^+)=\nu$. Since, in the present case, $F_\rho$ is a rational function, then this amounts to solving for the roots of the polynomial $\textstyle\int_0^\theta\lambda(\rho,s)ds-\nu\bar\lambda(\rho)=0$ and picking the only one in $\mathcal{P}$. When several parameters are involved, Markov Chain Monte Carlo (MCMC) methods can be used to sample from a multidimensional distribution. The procedure is not described here but the interested reader may look at the book \cite{Devroye:86}.

\section{Examples}\label{sec:ex}

The calculations are all performed on a PC equipped with 16GB of RAM and a processor Intel i7-5600U @ 2.60Ghz.

\subsection{Stability analysis - dimension 2}

We consider here the system \eqref{eq:mainsyst}-\eqref{eq:mainsystR}  with the matrix
\begin{equation}\label{eq:ex1}
  A(\rho)=\begin{bmatrix}
    0 & 1\\
    2-\rho & -1
  \end{bmatrix},
  %,E=\begin{bmatrix}
%    1\\0
%  \end{bmatrix},C=\begin{bmatrix}
%    0\\ 1
%  \end{bmatrix}^T,F,B,D=0
\end{equation}
all the other matrices being zero. We also pick $\mathcal{P}=[0,\bar\rho]$, $\lambda(\rho,\theta)=\lambda_0$. Using polynomials of order 2, we get the results in Table \ref{tab:stabalpha} where we can see that, as expected, increasing $\bar\rho$ while keeping $\lambda_0$ constant, or vice-versa, makes the system more and more stable. This is because of the fact that, in the present case, the matrix $A(\rho)-\lambda_0\bar{\rho}I$ gets more and more stable as the product $\lambda_0\bar{\rho}$ increases. Note that, In this case, the average time between two consecutive jumps is given by $1/(\lambda_0\bar{\rho})$ whereas the next value for the parameter is drawn from the uniform distribution $\mathcal{U}(0,\bar\rho)$. Computational-wise, the SOS program has 131 primal variables and 48 dual variables, and it takes 0.2 second to solve it.

\begin{table*}
\centering
\caption{Evolution of the computed exponential mean-square stability rate $\alpha$ for the system \eqref{eq:mainsyst}-\eqref{eq:mainsystR} using Theorem \ref{th:stab} and polynomials of order 2 for the system  \eqref{eq:ex1}.}\label{tab:stabalpha}
\begin{tabular}{|c||cccccc|}
  \hline
 \backslashbox{$\lambda_0$}{$\bar{\rho}$} &  5 & 6 & 7 & 8 & 9 & 10\\
  \hline
  \hline
 0.1 &infeasible   &infeasible    &0.3589    &1.8509    &3.2588    &4.7941\\
 0.25 &infeasible    &0.7345    &3.3211    &4.6144    &7.5190    &8.0626\\
 0.5 &1.4279    &3.2212    &5.1031    &6.1485   &10.3547   &11.9752\\
 0.75 &2.2643    &3.8085    &7.2430    &6.2499   &10.9252   &14.0590\\
 1 &   2.8106    &5.4874    &8.6958   &10.9797   &13.2820   &15.4967\\
   \hline
  \end{tabular}
  \end{table*}
  
  \begin{table*}
  \centering
\caption{Evolution of the computed exponential mean-square stability rate $\alpha$ for the system \eqref{eq:mainsyst}-\eqref{eq:mainsystR} using Theorem \ref{th:stab} and polynomials of order 2 for the system  \eqref{eq:ex2}.}\label{tab:stabalpha2}
\begin{tabular}{|c||ccc|}
  \hline
 \backslashbox{$\lambda_0$}{$\bar{\rho}$} & 1 & 2 & 3\\
  \hline
  \hline
 1 & 0.1857  &  0.1713 &  infeasible\\
 3  & 0.1801  &  0.1640 &  infeasible\\
 7  & 0.1827  &  0.1822    & 0.0489\\
 12  &  0.1856 &   0.1922  &  0.3234\\
 25 & 0.1884  &  0.2005    &0.3967\\
   \hline
   \end{tabular}
  \end{table*}

%\begin{minipage}{0.8\textwidth}
%\begin{table}
%%\centering
%  \caption{}\label{tab:stabalpha}
%
%  \end{table}
%\end{minipage}

\subsection{Stability analysis- dimension 4}

We consider here the system  \eqref{eq:mainsyst}-\eqref{eq:mainsystR}  with the matrix
\begin{equation}\label{eq:ex2}
  A(\rho)=\begin{bmatrix}
    \rho^2 & 1 & 1 & 0\\
    -2-\rho & -1 & 0 & 1\\
    1 & 0 & -3+\rho/2 & 0\\
    0 & 1 & 0 & -3+\rho/2
  \end{bmatrix},
\end{equation}
all the others being set to 0. We also pick the same parameters set and rate function for the parameters.  Using polynomials of order 2, we get the results in Table \ref{tab:stabalpha2} where we can observe a similar trend as in the previous example. Computational-wise, the number of primal/dual decision variables is 506/160 and it takes 0.3 second to solve the SOS program.
%
%\begin{table*}
%  \centering
%  \caption{Evolution of the computed exponential mean-square stability rate $\alpha$ for the system \eqref{eq:mainsyst}-\eqref{eq:mainsystR} using Theorem \ref{th:stab} and polynomials of order 2.}\label{tab:stabalpha2}
%  \begin{tabular}{|c||ccc|}
%  \hline
% \backslashbox{$\lambda_0$}{$\bar{\rho}$} & 1 & 2 & 3\\
%  \hline
%  \hline
% 1 & 0.1857  &  0.1713 &  infeasible\\
% 3  & 0.1801  &  0.1640 &  infeasible\\
% 7  & 0.1827  &  0.1822    & 0.0489\\
% 12  &  0.1856 &   0.1922  &  0.3234\\
% 25 & 0.1884  &  0.2005    &0.3967\\
%   \hline
%  \end{tabular}
%\end{table*}

\subsection{Stabilization by state-feedback}

We consider now the system \eqref{eq:mainsyst}-\eqref{eq:mainsystR}  with the matrices
\begin{equation}
  A(\rho)=\begin{bmatrix}
    3-\rho & 1\\
    1-\rho & 2+\rho
  \end{bmatrix},B=\begin{bmatrix}
    0\\1+\rho
  \end{bmatrix},C=\begin{bmatrix}
    0\\ 1
  \end{bmatrix}^T,
\end{equation}
$D=F=0$, and $\mathcal{P}=[0,\bar\rho]$. Choosing $\bar\rho=1$ and $\lambda(\rho,\theta)=100$ and $\gamma=1$, we get the controller

\begin{equation}
\small
  K(\rho)=\begin{bmatrix}
128\dfrac{-132.6\rho^4 + 3.441\rho^3 + 512.3\rho^2 + 1086\rho - 853.1}{- 74.3\rho^4 + 82.09\rho^3 + 896.2\rho^2 - 881.7\rho + 962.5}\\
  64\dfrac{1.692\rho^4 + 349.1\rho^3 - 874.9\rho^2 - 1614\rho + 909.4}{- 7.43\rho^4 + 8.209\rho^3 + 89.62\rho^2 - 88.17\rho + 96.25}
  \end{bmatrix}^T.
\end{equation}
For simulation purposes, we set $w(t)=10(H(t)-H(t-1))$ where $H$ is the Heavyside step function. We can see in Figure \ref{fig:OL_rho} (top) that the open-loop system is unstable. At the bottom of the same figure, a typical random trajectory for the parameter is depicted. Since $\lambda=100$, then the average time between two successive jumps is $1/\lambda$, (i.e. 10ms) and then the next value for the parameter is simply drawn from $\mathcal{U}(0,1)$. Multiple state responses of the closed-loop system to the initial condition $[-2, 4]$ are depicted in Figure \ref{fig:XnoW} whereas the responses to the input $w$ (with zero initial conditions) are depicted in Figure \ref{fig:XW}. The SOS program has 2290/603 primal and dual variables, and it takes about a second to solve. The reason why the number of variables is much larger than in the previous examples is because the size of the parameter-dependent LMI in Theorem \ref{eq:stabzL2} is larger than that in Theorem \ref{th:stab}.

\begin{figure}
  \centering
  \includegraphics[width=0.65\textwidth]{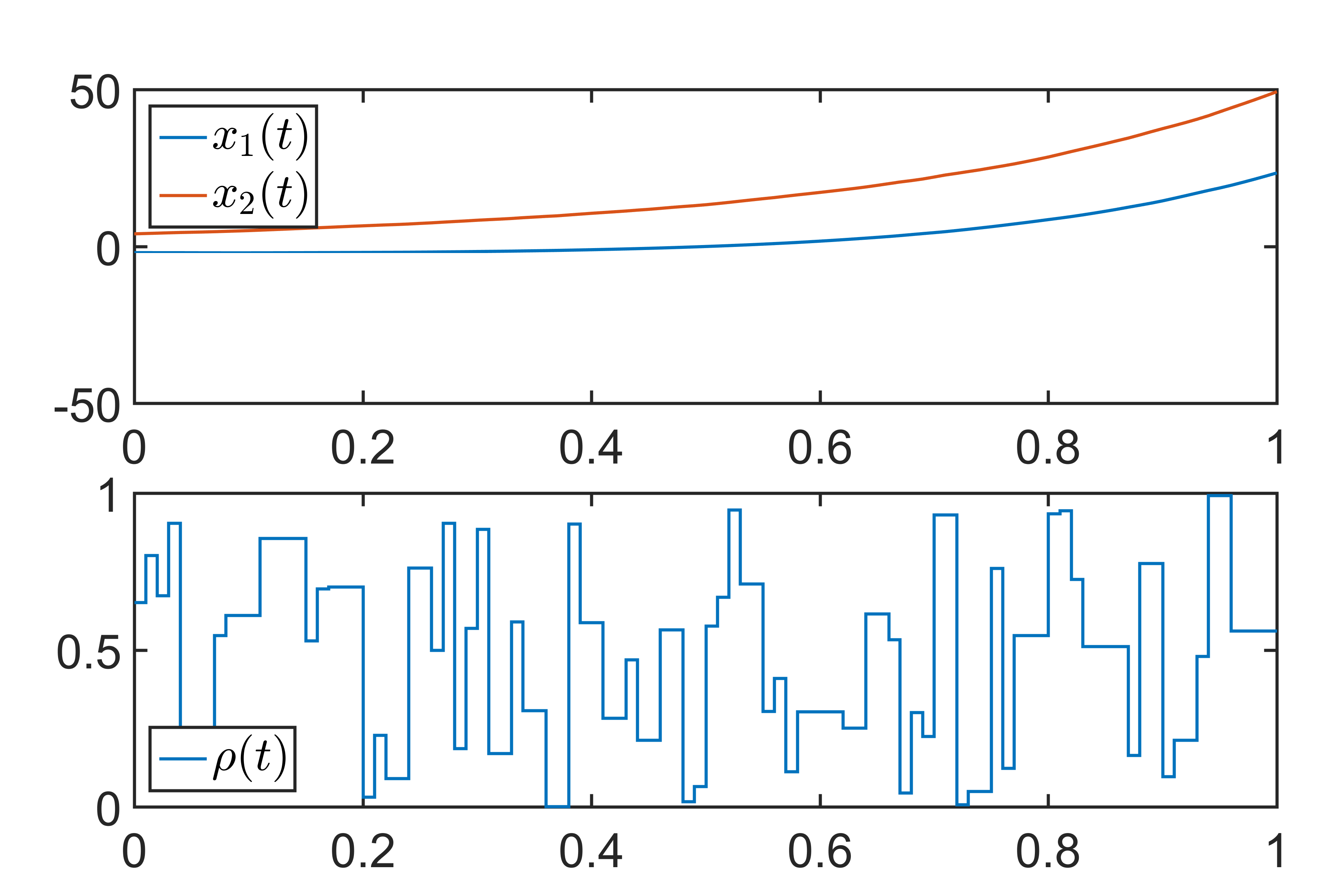}
  \caption{Evolution of the states of the open-loop system (top) and a typical trajectory for the parameter.}\label{fig:OL_rho}
\end{figure}

\begin{figure}
  \centering
  \includegraphics[width=0.65\textwidth]{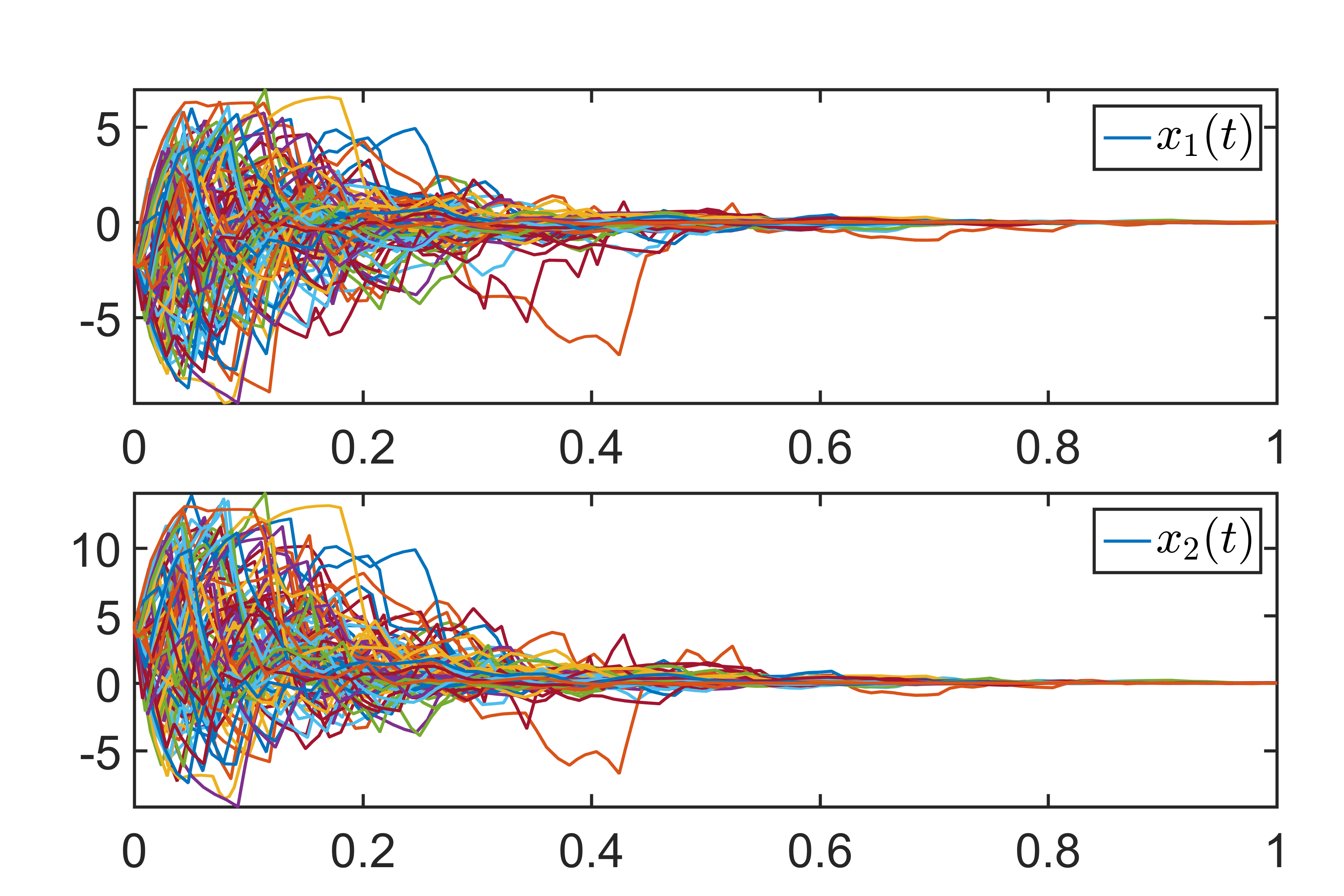}
  \caption{Evolution of the states of the closed-loop system for 100 realizations of the stochastic trajectories with no disturbance.}\label{fig:XnoW}
\end{figure}

\begin{figure}
  \centering
  \includegraphics[width=0.65\textwidth]{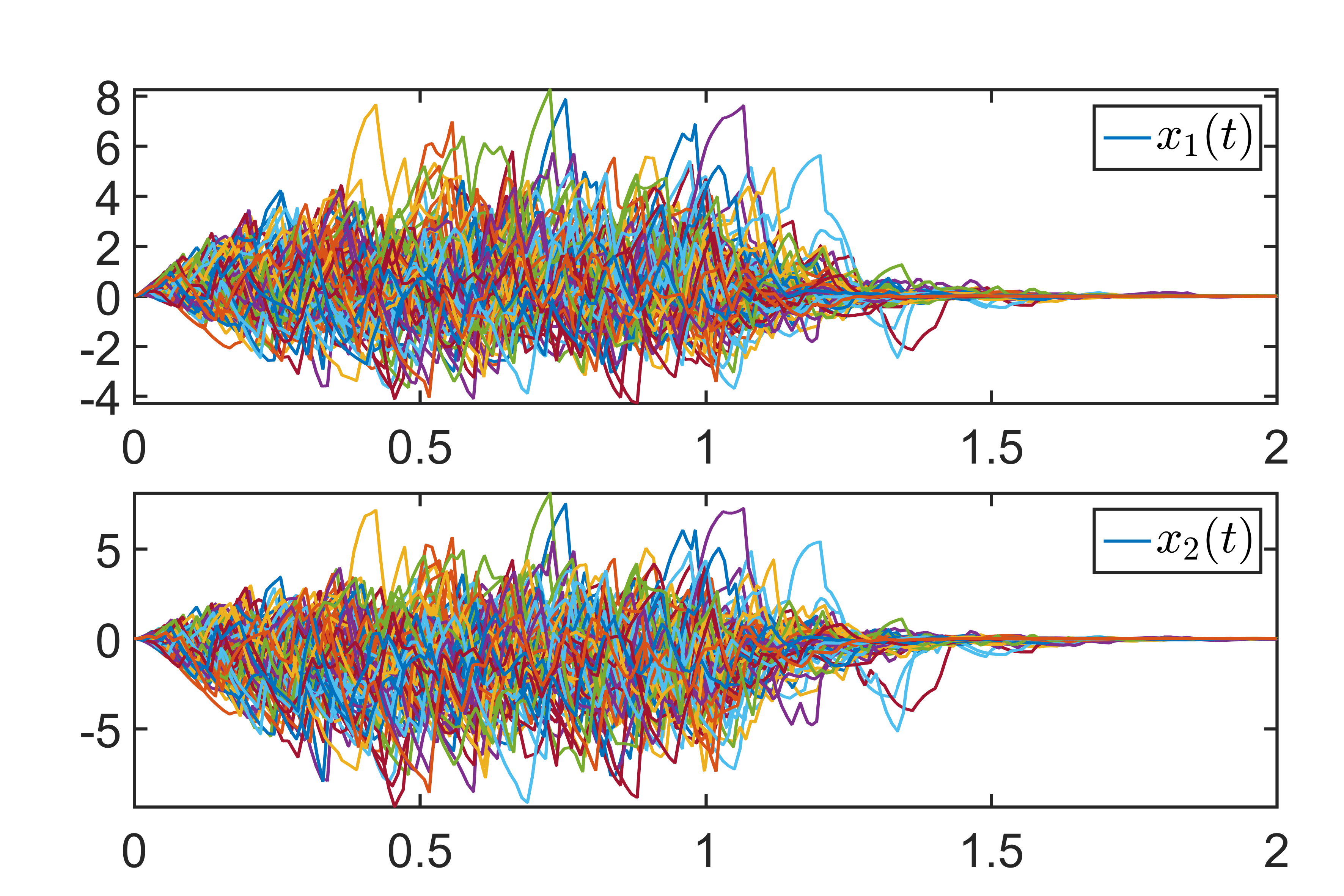}
  \caption{Evolution of the states of the closed-loop system for 100 realizations of the stochastic trajectories with disturbance.}\label{fig:XW}
\end{figure}

%\section{Discussion and Future Works}

%Sufficient stability and performance conditions as well as stabilization conditions via state-feedback have been obtained for LPV systems with piecewise constant parameters subject to spontaneous Poissonian jumps have been obtained for the first time. It has been shown that, under some assumption on the data of the systems and on the decision variables of the conditions, the resulting conditions could be solved using polynomial programming techniques.

\section{Future work}

A natural first question is to what extent the mean-square stability condition in Theorem \ref{th:stab} (and that of Theorem \ref{th:stabL2}) is also necessary. While it is well known that parameter-dependent quadratic Lyapunov functions are in general sufficient for the stability of deterministic LPV systems, this may not be the case here in hindsight of the existing results on Markov jump linear systems for which the Lyapunov function $V(x,\sigma)=x^TP(\sigma)x$ (where $\sigma$ is the mode of the Markov process) is necessary and sufficient for establishing the mean-square stability. A second question is whether convex conditions for the design of dynamic output-feedback controllers can be obtained in a nonconservative way.

%\bibliographystyle{IEEEtran}
%\bibliography{../../../../Lastbib/global,../../../../Lastbib/briat}


\end{document}